\documentclass[11pt]{article}
\usepackage{amsmath, amscd, amssymb, amsthm, verbatim}

\newtheorem{theorem}{Theorem}[section]

\newtheorem{corollary}[theorem]{Corollary}

\newtheorem{proposition}[theorem]{Proposition}

\theoremstyle{definition}
\newtheorem{definition}[theorem]{Definition}
\newtheorem{remark}[theorem]{Remark}
\newtheorem{example}[theorem]{Example}

\newcommand{\M}{M}
\newcommand{\A}{\mathcal{A}}
\newcommand{\F}{\mathcal{F}}
\newcommand{\N}{\mathbb{N}}
\newcommand{\R}{\mathbb{R}}
\newcommand{\OO}{\mathcal{O}}

\newcommand{\U}{\tilde{U}}
\newcommand{\X}{\tilde{X}}
\newcommand{\Z}{\tilde{Z}}
\newcommand{\nablaa}{\tilde{\nabla}}

\newcommand{\parr}{\slash \! \slash}
\newcommand{\tildeparr}{\skew8\widetilde{\slash\!\slash}^{\phantom{.}}}
\newcommand{\tildeparrinv}{\skew8\tilde{\slash\!\slash}^{-1}}

\DeclareMathOperator{\Ric}{Ric}

\DeclareMathOperator{\Hess}{Hess}
\DeclareMathOperator{\m}{m}
\DeclareMathOperator{\Ad}{Ad}
\DeclareMathOperator{\Id}{Id}
\DeclareMathOperator{\Riem}{Riem}
\DeclareMathOperator{\GL}{GL}

\begin{document}

\title{Martingales on manifolds with time-dependent connection}

 \date{}

\author{Hongxin Guo\footnote{School of Mathematics and Information Science, Wenzhou University, Wenzhou, Zhejiang 325035,
China. E-mail address: \texttt{guo@wzu.edu.cn}}, 
Robert Philipowski and Anton Thalmaier\footnote{Mathematics Research Unit, FSTC, University of Luxembourg, 6, rue Richard Coudenhove-Kalergi, L-1359 Luxembourg, Grand Duchy of Luxembourg. E-mail adresses: \texttt{robert.philipowski@uni.lu},
\texttt{anton.thalmaier@uni.lu}}}


\renewcommand{\theequation}{\arabic{section}.\arabic{equation}}
\maketitle

\begin{abstract}
We define martingales on manifolds with time-dependent connection, extending in this way the theory of stochastic processes on
manifolds with time-changing geometry initiated by Arnaudon, Coulibaly and Thalmaier (2008). 
We show that some, but not all properties of martingales on manifolds with a fixed connection extend to this more general
setting.
\end{abstract}

\section{Introduction}
Stochastic analysis on manifolds with a fixed connection or a fixed
Riemannian metric has been studied for a long time, see e.g.\ the
books by Hackenbroch and Thalmaier~\cite{hackenbroch} and
Hsu~\cite{hsu}. Motivated by Perelman's proof of the geometrization
and hence the Poincar\'e conjecture using Ricci flow~\cite{perelman1,
  perelman2,
  perelman3}, 
Arnaudon, Coulibaly and Thalmaier~\cite{ACT} introduced Brownian
motion on a manifold with a time-dependent Riemannian metric. Thanks
to the subsequent papers by Coulibaly-Pasquier~\cite{Coulibaly},
Kuwada and Philipowski~\cite{kuwadaphilipowski, kuwadaphilipowski2}
and Paeng~\cite{paeng}, Brownian motion in such a time-dependent
framework is now well understood.

Stochastic analysis on manifolds, however, is not restricted to the
study of Brownian motion. Another important topic is martingale
theory, which in the case of a fixed connection is treated in depth in
e.g.\ \cite{emery, hackenbroch, hsu}, but which has not yet been
studied in the case of a time-dependent
connection. 
The aim of the present paper is to fill this gap.

\section{Horizontal lift, stochastic parallel transport and stochastic
  development 
  on manifolds with time-dependent connection}

Let $\M$ be a $d$-dimensional differentiable manifold, $\pi: \F(M) \to
M$ the frame bundle and $(\nabla(t))_{t \geq 0}$ a family of linear
connections on $M$ depending smoothly on
$t$. 
Let $(\Omega, \F, P, (\F_t)_{t \geq 0})$ be a filtered probability
space. Throughout the whole paper, the notions of martingale,
semimartingale, etc.\ are understood with respect to this
filtration. Moreover, all processes are tacitly assumed to be
continuous.

\begin{definition}[cf.\ {\cite[Definition~7.135]{hackenbroch}} for the case of a fixed connection]
  An $\F(M)$-valued semimartingale $U$ is said to be {\em
    $(\nabla(t))_{t \geq 0}$-horizontal} if
  \begin{equation}\label{hor1a}
    \omega_t({}\circ dU_t) = 0,
  \end{equation}
  where $\omega_t$ is the $\R^{d \times d}$-valued connection form
  with respect to $\nabla(t)$.
\end{definition}

\begin{proposition}[cf.\ {\cite[Satz~7.141]{hackenbroch}} for the case of a fixed connection]\label{lift}
  Let $X$ be an $M$-valued semimartingale and $U_0$ an
  $\F_0$-measurable $\F(M)$-valued random variable with $\pi U_0 =
  X_0$. Then there exists a unique $(\nabla(t))_{t \geq 0}$-horizontal
  lift $U$ of $X$ starting at $U_0$, i.e.\ an $\F(M)$-valued
  semimartingale satisfying \eqref{hor1a} and $\pi U = X$. Moreover,
  starting with an arbitrary lift $\U$ of $X$ satisfying $\U_0 = U_0$,
  the horizontal lift $U$ can be constructed in the following way: Let
  \begin{equation}\label{gamma}
    \gamma_t := \int_0^t \omega_s({}\circ d\U_s),
  \end{equation}
  and $G$ the solution of the $\GL_d(\R)$-valued SDE
  \begin{equation}\label{gproc}
    d G_t = -\sum_{\alpha,\beta=1}^d E_{\alpha \beta} G_t \circ d\gamma_t^{\alpha \beta}, \qquad G_0 = I,
  \end{equation}
  where $ E_{\alpha \beta} \in \R^{d \times d}$ is the matrix whose
  $(ij)$-entry is $1$ if $i = \alpha$ and $j = \beta$, and $0$
  otherwise.  Then
  \begin{equation}\label{ug}
    U_t = \U_t  G_t.
  \end{equation}
\end{proposition}

\begin{proof}
  We first show that the process $U$ defined by \eqref{ug} is indeed
  $(\nabla(t))_{t \geq 0}$-horizontal. Letting $\Phi: \F(M) \times
  \GL_d(\R) \to \F(M)$ defined by $\Phi(u,g) := ug$, we have
  \[
  (\Phi^*\omega_t)_{(u,g)} = (R_g^*\omega_t)_u + \theta_g,
  \]
  where $\theta_g := dL_g^{-1}$ ($L_g$ and $R_g$ denoting left resp.\
  right multiplication with $g$). Since moreover by
  \cite[Bemerkung~7.128~(ii)]{hackenbroch}, $R_g^* \omega_t =
  \Ad(g^{-1}) \omega_t$, we obtain
  \begin{align*}
    \omega_t({}\circ dU_t)
    & =  \omega_t({}\circ d\Phi(\U_t, G_t))\\
    & =  (\Phi^* \omega_t)({}\circ d(\U_t, G_t))\\
    & =  (R_{G_t}^* \omega_t)({}\circ d\U_t) + \theta({}\circ dG_t)\\
    & =  \Ad(G_t^{-1}) \omega_t({}\circ d\U_t) + dL_{G_t}^{-1}({}\circ dG_t)\\
    & = \Ad(G_t^{-1}) {}\circ d \gamma_t - \sum_{\alpha,\beta=1}^d
    G_t^{-1} E_{\alpha \beta} G_t \circ d\gamma_t^{\alpha \beta} = 0.
  \end{align*}
  To show uniqueness assume that $U'$ is another $(\nabla(t))_{t \geq
    0}$-horizontal lift of $X$ with $U'_0 = U_0$. Then $U = U' G$ with
  a $\GL^d(\R)$-valued semimartingale $G = (G_t)_{t \geq 0}$ starting
  at $I$. The above computation yields $dL_{G_t}^{-1}({}\circ dG_t) =
  0$ and hence $dG_t = 0$, so that $G_t = I$ for all $t \geq 0$.
\end{proof}

\begin{definition}[cf.\ {\cite[Definition~7.144]{hackenbroch}} for the case of a fixed connection]
  The {\em $(\nabla(t))_{t \geq 0}$-parallel transport} along an
  $M$-valued semimartingale $X$ is the family of isomorphisms
  $\parr_{s,t}: T_{X_s}M \to T_{X_t}M$ ($0 \leq s \leq t$) defined by
  \[
  \parr_{s,t} := U_t U_s^{-1},
  \]

  where 
  $U$ is an arbitrary $(\nabla(t))_{t \geq 0}$-horizontal lift of
  $X$. (As in the case of a fixed connection, the result does not
  depend on the choice of the horizontal lift.)
\end{definition}

\begin{definition}[cf.\ {\cite[Definition~7.136]{hackenbroch}} for the case of a fixed connection]
  Let $X$ be an $M$-valued semimartingale, $U_0$ an $\F_0$-measurable
  $\F(M)$-valued random variable with $\pi U_0 = X_0$, and $U$ the
  unique $(\nabla(t))_{t \geq 0}$-horizontal lift of $X$ starting at
  $U_0$. The $\R^d$-valued process
  \begin{equation*}
    Z_t := \int_0^t \vartheta({}\circ dU_s)
  \end{equation*}
  is called the $(\nabla(t))_{t \geq 0}${\em-antidevelopment} of $X$
  (or $U$) with initial frame $U_0$; here $\vartheta$ is the canonical
  $\R^d$-valued 1-form on $\F(M)$,
$$ \vartheta_u(w)=u^{-1}(d\pi w),\quad w\in T_u\F(M).$$
\end{definition}

\begin{remark}\label{recov}
  A $(\nabla(t))_{t \geq 0}$-horizontal semimartingale $U$ can be
  recovered from its $(\nabla(t))_{t\geq 0}$-antidevelopment $Z$ and
  its initial value $U_0$ as the solution of the SDE
  \begin{equation*}
    dU_t = \sum_{i=1}^d H_i^{\nabla(t)}(U_t) \circ dZ_t^i,
  \end{equation*}
  where $(H_i^{\nabla(t)})_{i=1}^d$ are the standard
  $\nabla(t)$-horizontal vector fields on $\F(M)$, i.e.,
  \begin{equation*}
    H_i^{\nabla(t)}(u)=h^{\nabla(t)}_u(ue_i),\quad u\in\F(M),
  \end{equation*}
  with $h_u^{\nabla(t)}\colon\,T_{\pi(u)}M\to T_u\F(M)$ of the
  connection $\nabla(t)$.
\end{remark}

\begin{proof}
  Let $V$ be the solution of the SDE
  \[
  dV_t = \sum_{i=1}^d H_i^{\nabla(t)}(V_t) \circ dZ_t^i, \qquad V_0 =
  U_0.
  \]
  Then $V$ is a $(\nabla(t))_{t \geq 0}$-horizontal lift of $\pi U$
  with the same initial value as $U$, hence $V = U$.
\end{proof}

\begin{corollary}
  \label{ito1}
  Let $U$ be a $(\nabla(t))_{t \geq 0}$-horizontal semimartingale, and
  $X := \pi U$. Then we have the following It\^o formulas:
  \begin{enumerate}
  \item For all smooth functions $f$ on $\R_+ \times \F(M)$ we have
    \begin{align*}
      d f(t, U_t) & =  \frac{\partial f}{\partial t}(t, U_t) dt + \sum_{i=1}^d H_i^{\nabla(t)} \!  f(t,U_t) \circ d Z_t^i \\
      & = \frac{\partial f}{\partial t}(t, U_t) dt + \sum_{i=1}^d
      H_i^{\nabla(t)} \!  f(t,U_t) \,d Z_t^i + \frac{1}{2}
      \sum_{i,j=1}^d H_i^{\nabla(t)} \! H_j^{\nabla(t)} \!  f(t,U_t)
      \,d \langle Z^i,Z^j \rangle_t.
    \end{align*}

  \item For all smooth functions $f$ on $\R_+ \times M$ we have
    \begin{align*}
      d f(t, X_t) & =  \frac{\partial f}{\partial t}(t, X_t) dt + \sum_{i=1}^d (U_t e_i) f(t,X_t)  \circ d Z_t^i \\
      & = \frac{\partial f}{\partial t}(t, X_t) dt + \sum_{i=1}^d (U_t
      e_i) f(t,X_t) \,d Z_t^i + \frac{1}{2} \sum_{i,j=1}^d
      \Hess^{\nabla(t)} \! f(U_t e_i, U_t e_j) \,d \langle Z^i,Z^j
      \rangle_t.
    \end{align*}
  \end{enumerate}
\end{corollary}

\begin{remark}\label{dev}
  In the situation of Proposition~\ref{lift} let $\Z_t := \int_0^t
  \vartheta({}\circ d \U_s)$. Then
  \begin{equation}\label{devfor}
    dZ_t = G_t^{-1} \circ d \Z_t.
  \end{equation}
\end{remark}

\begin{proof}
  Since $\pi U_t = \pi \U_t$, we have
  \[
  dZ_t = \vartheta({}\circ d U_t) = U_t^{-1} \circ \pi_* d U_t =
  G_t^{-1} \U_t^{-1} \circ \pi_* d \U_t = G_t^{-1} \vartheta({}\circ d
  \U_t) = G_t^{-1} \circ d \Z_t. \qedhere
  \]
\end{proof}

\section{Alternative definition of horizontality in the Riemannian case}
In this section we assume that for each $t \geq 0$ the connection
$\nabla(t)$ is the Levi-Civita connection of a Riemannian metric
$g(t)$ depending smoothly on $t$ (we call this the Riemannian
case). In this situation it seems natural to require that each $U_t$
takes values in the $g(t)$-orthonormal frames of~$M$, i.e. $U_t \in
\OO_{g(t)}(M)$ for all $t \geq 0$. To ensure this, one has to add a
correction term to~\eqref{hor1a}.

\begin{definition}
  An $\F(M)$-valued semimartingale $U$ is said to be {\em $(g(t))_{t
      \geq 0}$-Riemann-horizontal} if $U_0 \in \OO_{g(0)}(M)$ and
  \begin{equation}\label{horriema}
    \omega_t({}\circ dU_t) 
    = - \frac{1}{2} \sum_{\alpha, \beta = 1}^d \frac{\partial g}{\partial t}(t, U_te_\alpha, U_t e_\beta) E_{\alpha \beta}dt.
  \end{equation}
\end{definition}

In Proposition~\ref{orthonormal} below we will show that any
$(g(t))_{t \geq 0}$-Riemann-horizontal semimartingale $U$ satisfies
indeed $U_t \in \OO_{g(t)}(M)$ for all $t \geq 0$.  Before doing so we
show that the results of the previous section carry over to $(g(t))_{t
  \geq 0}$-Riemann-horizontal processes with appropriate
modifications:

\begin{proposition}\label{riemhor2}
  Let $X$ be an $M$-valued semimartingale and $U_0$ an
  $\F_0$-measurable $\OO_{g(0)}(M)$-valued random variable with $\pi
  U_0 = X_0$. Then there exists a unique $(g(t))_{t \geq
    0}$-Riemann-horizontal lift $U$ of $X$ starting at $U_0$, i.e.\ an
  $\F(M)$-valued semimartingale satisfying \eqref{horriema} and $\pi U
  = X$. Moreover, starting with an arbitrary lift $\U$ of $X$
  satisfying $\U_0 = U_0$, the $(g(t))_{t \geq 0}$-Riemann-horizontal
  lift $U$ can be constructed in the following way: Let
  \begin{equation}\label{gamma2}
    \gamma_t := \int_0^t \omega_s({}\circ d\U_s),
  \end{equation}
  and $G$ the solution of the $\GL_d(\R)$-valued SDE
  \begin{equation}\label{gproc2}
    d G_t = -\sum_{\alpha,\beta=1}^d E_{\alpha \beta} G_t \circ d\gamma_t^{\alpha \beta} - \frac{1}{2} \sum_{\alpha, \beta = 1}^d \frac{\partial g}{\partial t}(t, \tilde U_t G_t e_\alpha, \tilde U_t G_t e_\beta) G_t E_{\alpha \beta} dt, \qquad G_0 = I,
  \end{equation}
  Then
  \begin{equation}\label{ug2}
    U_t = \U_t  G_t.
  \end{equation}
\end{proposition}

\begin{proof}
  We first show that the process $U$ defined by \eqref{ug2} is indeed
  $(g(t))_{t \geq 0}$-Riemann-horizontal. As in the proof of
  Proposition~\ref{lift} we obtain
  \begin{align*}
    \omega_t({}\circ dU_t) & =  (R_{G_t}^* \omega_t)({}\circ d\U_t) + \theta({}\circ dG_t)\\
    & =  \Ad(G_t^{-1}) \omega_t({}\circ d\U_t) + dL_{G_t}^{-1}({}\circ dG_t)\\
    & =  \Ad(G_t^{-1}) \circ d \gamma_t - \sum_{\alpha,\beta=1}^d G_t^{-1} E_{\alpha \beta} G_t \circ d\gamma_t^{\alpha \beta} - \frac{1}{2} \sum_{\alpha, \beta = 1}^d \frac{\partial g}{\partial t}(t, \tilde U_t G_t e_\alpha, \tilde U_t G_t e_\beta) E_{\alpha \beta} dt \\
    & = - \frac{1}{2} \sum_{\alpha, \beta = 1}^d \frac{\partial
      g}{\partial t}(t, U_t e_\alpha, U_t e_\beta) E_{\alpha \beta}
    dt.
  \end{align*}
  Uniqueness of $U$ can be proved in the same way as in
  Proposition~\ref{lift}.
\end{proof}

\begin{definition}
  The {\em $(g(t))_{t \geq 0}$-Riemann-parallel transport} along an
  $M$-valued semimartingale $X$ is the family of isomorphisms
  $\parr_{s,t}: T_{X_s}M \to T_{X_t}M$ ($0 \leq s \leq t$) defined by
  \[
  \parr_{s,t} := U_t U_s^{-1},
  \]
  where $U$ is an arbitrary $(g(t))_{t \geq 0}$-Riemann-horizontal
  lift of $X$.
\end{definition}

\begin{definition}
  Let $X$ be an $M$-valued semimartingale, $U_0$ an $\F_0$-measurable
  $\OO_{g(0)}(M)$-valued random variable with $\pi U_0 = X_0$, and $U$
  the unique $(g(t))_{t \geq 0}$-Riemann-horizontal lift of $X$
  starting at $U_0$. The $\R^d$-valued process
  \[
  Z_t := \int_0^t \vartheta({}\circ dU_s)
  \]
  is called the $(g(t))_{t \geq 0}$-Riemann-antidevelopment of $X$ (or
  $U$) with initial frame $U_0$.
\end{definition}

\begin{remark}
  A $(g(t))_{t \geq 0}$-Riemann-horizontal process $U$ can be
  recovered from its $(g(t))_{t \geq 0}$-antidevelopment $X$ and its
  initial value $U_0$ as the solution of the SDE
  \[
  dU_t = \sum_{i=1}^d H_i^{\nabla(t)}(U_t) \circ dZ_t^i - \frac{1}{2}
  \sum_{\alpha, \beta = 1}^d \frac{\partial g}{\partial t}(t,
  U_te_\alpha, U_t e_\beta) V_{\alpha \beta}(U_t) dt,
  \]
  where $(V_{\alpha, \beta})_{\alpha, \beta = 1}^d$ are the canonical
  vertical vector fields defined as
  \[
  V^{\alpha \beta} f(u) = \left. \frac{d}{ds} \right|_{s=0} f(u (I +
  sE_{\alpha \beta}))
  \]
  ($I$ denoting the identity matrix).
\end{remark}

\begin{proof}
  Noting that $\omega_t(V_{\alpha \beta}) = E_{\alpha \beta}$ (by the
  definition of $\omega_t$), this can be proved in the same way as
  Remark~\ref{recov}.
\end{proof}

\begin{corollary}
  \label{ito2}
  Let $U$ be a $(g(t))_{t \geq 0}$-Riemann-horizontal semimartingale,
  and $X := \pi U$. Then we have the following It\^o formulas:
  \begin{enumerate}
  \item For all smooth functions $f$ on $\R_+ \times \F(M)$ we have
    \begin{align}
      d f(t, U_t) & =  \frac{\partial f}{\partial t}(t, U_t) dt + \sum_{i=1}^d H_i^{\nabla(t)} \!  f(t,U_t) \circ d Z_t^i  - \frac{1}{2} \sum_{\alpha, \beta = 1}^d \frac{\partial g}{\partial t}(t, U_t e_\alpha, U_t e_\beta) V_{\alpha \beta} f (t, U_t) dt \nonumber\\
      & =  \frac{\partial f}{\partial t}(t, U_t) dt + \sum_{i=1}^d H_i^{\nabla(t)} \!  f(t,U_t) d Z_t^i + \frac{1}{2} \sum_{i,j=1}^d H_i^{\nabla(t)} \! H_j^{\nabla(t)} \!  f(t,U_t) d \langle Z^i,Z^j \rangle_t \nonumber \\
      &\quad - \frac{1}{2} \sum_{\alpha, \beta = 1}^d \frac{\partial
        g}{\partial t}(t, U_t e_\alpha, U_t e_\beta) V_{\alpha \beta}
      f (t, U_t) dt. \label{itoformula}
    \end{align}

  \item For all smooth functions $f$ on $\R_+ \times M$ we have
    \begin{align*}
      d f(t, X_t) & =  \frac{\partial f}{\partial t}(t, X_t) dt + \sum_{i=1}^d (U_t e_i) f(t,X_t)  \circ d Z_t^i \\
      & = \frac{\partial f}{\partial t}(t, X_t) dt + \sum_{i=1}^d (U_t
      e_i) f(t,X_t) d Z_t^i + \frac{1}{2} \sum_{i,j=1}^d
      \Hess^{\nabla(t)} \!\! f(U_t e_i, U_t e_j) d \langle Z^i,Z^j
      \rangle_t.
    \end{align*}
  \end{enumerate}
\end{corollary}

\begin{proposition}
  \label{orthonormal}
  Let $U$ be a $(g(t))_{t \geq 0}$-Riemann-horizontal semimartingale.
  If $U_0 \in \OO_{g(0)}(M)$, then $U_t \in \OO_{g(t)}(M)$ for all $t
  \geq 0$.
\end{proposition}

\begin{proof}
  We have to show that $\langle U_t e_i, U_t e_j \rangle_{g(t)}$ is
  constant for all $i,j \in \{1, \ldots, d\}$.  To do so we fix $i,j
  \in \{1, \ldots, d\}$ and apply It\^o's formula~\eqref{itoformula}
  to the function $f(t,u) := \langle u e_i, u e_j
  \rangle_{g(t)}$. Obviously,
  \[
  \frac{\partial f}{\partial t}(t, u) = \frac{\partial g}{\partial
    t}(u e_i, u e_j).
  \]
  Since $f$ is constant along horizontal curves in $\F(M)$, we have
  \[
  H_i^{\nabla(t)} f = H_i^{\nabla(t)} H_j^{\nabla(t)} f = 0.
  \]
  Finally, for $u \in \OO_{g(t)}(M)$,
  \begin{align*}
    V^{\alpha \beta} f(t,u) & =  \left. \frac{d}{ds} \right|_{s=0} f(t, u (I + sE_{\alpha \beta}))\\
    & =  \left. \frac{d}{ds} \right|_{s=0} \langle u (I + sE_{\alpha \beta}) e_i, u (I + sE_{\alpha \beta}) e_j \rangle_{g(t)}\\
    & =  \left. \frac{d}{ds} \right|_{s=0} \langle (I + sE_{\alpha \beta}) e_i, (I + sE_{\alpha \beta}) e_j \rangle_{\R^d}\\
    & =  \langle E_{\alpha \beta} e_i, e_j \rangle_{\R^d} + \langle e_i, E_{\alpha \beta} e_j \rangle_{\R^d}\\
    & = \begin{cases}
      2 & \mbox{if } \alpha = \beta = i = j,\\
      1 & \mbox{if } i \neq j \mbox{ and } (\alpha = i, \beta = j \mbox{ or } \alpha = j, \beta = i),\\
      0 & \mbox{otherwise},
    \end{cases}
  \end{align*}
  so that
  \[
  \frac{1}{2} \sum_{\alpha, \beta = 1}^d \frac{\partial g}{\partial
    t}(t, u e_\alpha, u e_\beta) V_{\alpha \beta} f (t, u) =
  \frac{\partial g}{\partial t}(ue_i, ue_j) = - \frac{\partial
    f}{\partial t}(t,u). \qedhere
  \]
\end{proof}

\begin{remark}\label{dev2}
  In the situation of Proposition~\ref{riemhor2} let $\Z_t := \int_0^t
  \vartheta({}\circ d \U_s)$. Then
  \begin{equation}\label{devfor2}
    dZ_t = G_t^{-1} \circ d \Z_t.
  \end{equation}
\end{remark}

\begin{proof}
  This can be proved in the same way as Remark~\ref{dev}.
\end{proof}

\begin{remark}\label{relpar}
  Let $X$ be an $M$-valued semimartingale and $U_0$ an
  $\F_0$-measurable $\OO_{g(0)}(M)$-valued random variable with $\pi
  U_0 = X_0$. Then $X$ has on the one hand a unique $(\nabla(t))_{t
    \geq 0}$-horizontal lift $U$ starting at $U_0$, $(\nabla(t))_{t
    \geq 0}$-parallel transports $\parr_{s,t}$ ($0 \leq s \leq t$) and
  a $(\nabla(t))_{t \geq 0}$-anti\-de\-ve\-lop\-ment $Z$, and on the
  other hand a unique $(g(t))_{t \geq 0}$-Riemann-horizontal lift
  $U^{\Riem}$ starting at $U_0$, $(g(t))_{t \geq 0}$-Riemann-parallel
  transports $\parr_{s,t}^{\Riem}$ ($0 \leq s \leq t$) and a
  $(g(t))_{t \geq 0}$-Riemann-anti\-de\-ve\-lop\-ment $Z^{\Riem}$.
  Proposition~\ref{riemhor2} implies that
  \[
  d \left( U_t^{-1} U_t^{\Riem} \right) = -\frac{1}{2} U_t^{-1} \left(
    \frac{\partial g}{\partial t} \right)^\# U_t^{\Riem}dt
  \]
  and
  \[
  d \left( \parr_{0,t}^{-1} \parr_{0,t}^{\Riem} \right) =
  -\frac{1}{2} \parr_{0,t}^{-1} \left( \frac{\partial g}{\partial t}
  \right)^\# \parr_{0,t}^{\Riem}dt.
  \]
  Moreover in this case the process~$\gamma$ defined in \eqref{gamma}
  resp.\ \eqref{gamma2} and therefore also the process~$G$ defined in
  \eqref{gproc} resp.\ \eqref{gproc2} is of finite variation, so that
  the Stratonovich differential appearing in \eqref{devfor} resp.\
  \eqref{devfor2} may be replaced by an It\^o differential.
\end{remark}

\section{Quadratic variation and integration of
  1-forms} 

\begin{proposition}
  Let $X$ be an $M$-valued semimartingale, $U$ a $(\nabla(t))_{t \geq
    0}$-horizontal or $(g(t))_{t \geq 0}$-Riemann-horizontal lift of
  $X$, and $Z_t := \int_0^t \vartheta({}\circ dU_s)$ the corresponding
  $(\nabla(t))_{t \geq 0}$-antidevelopment resp.\ $(g(t))_{t \geq
    0}$-Riemann-antidevelopment. Then for every adapted $T^*M \otimes
  T^*M$-valued process~$B$ above~$X$ (i.e.\ $B_t \in T_{X_t}^*M
  \otimes T_{X_t}^*M$ for all $t \geq 0$) we have
  \[
  \int_0^t B_s (dX_s, dX_s) = \sum_{i,j=1}^d \int_0^t B_s(U_s e_i, U_s
  e_j) \, d \langle Z^i, Z^j \rangle_s.
  \]
\end{proposition}

\begin{proof}
  By \cite[Lemma~7.56~(iv)]{hackenbroch} there exist $\ell \in \N$,
  real-valued adapted processes $(B^{\mu \nu})_{\mu,\nu=1}^\ell$ and
  functions $h_1, \ldots, h_\ell \in C^\infty(M)$ such that $ B_t =
  \sum_{\mu, \nu=1}^\ell B_t^{\mu \nu} (dh_\mu \otimes dh_\nu)(X_t) $
  for all $t \geq 0$. It follows that
  \begin{align*}
    \int_0^t B_s (dX_s, dX_s)
    & =  \sum_{\mu, \nu=1}^\ell \int_0^t (B_s^{\mu \nu} dh_\mu \otimes dh_\nu) (dX_s, dX_s)\\
    & = \sum_{\mu, \nu=1}^\ell \int_0^t B_s^{\mu \nu} d \langle
    h_\mu(X), h_\nu (X) \rangle_s.
  \end{align*}
  Since by It\^o's formula (Corollary~\ref{ito1} resp.\
  Corollary~\ref{ito2})
  \begin{align*}
    d \langle h_\mu(X), h_\nu (X) \rangle_s
    & = \sum_{i,j=1}^d (U_s e_i) h_\mu(X_s) (U_s e_j) h_\nu(X_s) \, d \langle Z^i, Z^j \rangle_s\\
    & = \sum_{i,j=1}^d (dh_\mu \otimes dh_\nu)(U_s e_i, U_s e_j) \, d
    \langle Z^i, Z^j \rangle_s,
  \end{align*}
  the claim follows.
\end{proof}

By choosing $B_s = \Hess^{\nabla(s)} \! f(X_s)$ or (in the Riemannian
case) $B_s = g(s,X_s)$ we obtain the following two corollaries:

\begin{corollary}\label{ito3}
  For all smooth functions $f$ on $\R_+ \times M$ we have
  \begin{equation*}
    df(t,X_t) = \frac{\partial f}{\partial t}(t,X_t) dt + \sum_{i=1}^d (U_t e_i) f(t,X_t) d Z_t^i + \frac{1}{2} \Hess^{\nabla(t)} \!\! f(dX_t, dX_t). 
  \end{equation*}
\end{corollary}

\begin{corollary}
  In the Riemannian case,
  \[
  \int_0^t g(s)(dX_s, dX_s) = \sum_{i=1}^d \langle Z^i,Z^i \rangle_t.
  \]
\end{corollary}

\begin{proposition}
  Let $X$ be an $M$-valued semimartingale, $U$ a $(\nabla(t))_{t \geq
    0}$-horizontal or $(g(t))_{t \geq 0}$-Riemann-horizontal lift of
  $X$, and $Z_t := \int_0^t \vartheta({}\circ dU_s)$ the corresponding
  $(\nabla(t))_{t \geq 0}$-antidevelopment resp.\ $(g(t))_{t \geq
    0}$-Riemann-antidevelopment. Then for every adapted $T^*M$-valued
  process $\Psi$ above $X$ (i.e.\ $\Psi_t \in T_{X_t}^*M$ for all $t
  \geq 0$) we have
  \[
  \int_0^t \Psi_s ({}\circ dX_s) = \sum_{i=1}^d \int_0^t \Psi_s(U_s
  e_i) \circ d Z_s^i.
  \]
\end{proposition}

\begin{proof}
  By \cite[Lemma~7.56~(v)]{hackenbroch} there exist $\ell \in \N$,
  real-valued adapted processes $\Psi^1, \ldots, \Psi^\ell$ and
  functions $h_1, \ldots, h_\ell \in C^\infty(M)$ such that $\Psi_t =
  \sum_{\nu=1}^\ell \Psi_t^\nu dh_\nu(X_t)$ for all $t \geq 0$. It
  follows that
  \begin{align*}
    \int_0^t \Psi_s ({}\circ dX_s)
    & = \sum_{\nu=1}^\ell \int_0^t (\Psi_s^\nu dh_\nu) ({}\circ dX_s)\\
    & = \sum_{\nu=1}^\ell \int_0^t \Psi_s^\nu \circ dh_\nu (X_s).
  \end{align*}
  Since by It\^o's formula (Corollary~\ref{ito1} resp.\
  Corollary~\ref{ito2})
  \[
  d h_\nu (X_s) 
  = \sum_{i=1}^d dh_\nu(U_s e_i) \circ d Z_s^i,
  \]
  the claim follows.
\end{proof}

\section{Martingales on manifolds with time-dependent connection}

\begin{proposition}[cf.\ {\cite[Satz~7.147~(i)]{hackenbroch}} for the case of a fixed connection]\label{marteq}
  Let $X$ be an $M$-valued semimartingale. Then the following
  conditions are equivalent:
  \begin{enumerate}
  \item The $(\nabla(t))_{t \geq 0}$-antidevelopment of $X$ is an
    $\R^d$-valued local martingale.

  \item For all smooth $f: M \to \R$ the process
    \[
    f(X_t) - f(X_0) - \frac{1}{2} \int_0^t \Hess^{\nabla(s)} \!\!
    f(dX_s, dX_s)
    \]
    is a real-valued local martingale.
  \end{enumerate}
  Moreover, in the Riemannian case, these conditions are equivalent to
  the condition that the $(g(t))_{t \geq 0}$-Riemann-antidevelopment
  of $X$ is an $\R^d$-valued local martingale.
\end{proposition}

\begin{definition}\label{martingale}
  $X$ is called a {\em $(\nabla(t))_{t \geq 0}$-martingale} if the
  equivalent conditions of Proposition~\ref{marteq} are satisfied.
\end{definition}

\begin{proof}[Proof of Proposition~\rm\ref{marteq}]
  Let $Z$ be the $(\nabla(t))_{t \geq 0}$-antidevelopment or
  $(g(t))_{t \geq 0}$-Riemann-anti\-deve\-lop\-ment of $X$, and $f \in
  C^\infty(M)$. Then by Corollary~\ref{ito3}
  \[
  f(X_t) - f(X_0) - \frac{1}{2} \int_0^t \Hess^{\nabla(s)} \!\!
  f(dX_s, dX_s) = \sum_{i = 1}^d \int_0^t (U_s e_i) f(X_s) d Z_s^i.
  \]
  This is a local martingale for all $f \in C^\infty(M)$ if and only
  if $Z$ is an $\R^d$-valued local martingale.
\end{proof}

\begin{proposition}[Local expression]\label{local}
  A semimartingale $X$ is a $(\nabla(t))_{t \geq 0}$-martingale if and
  only if in local coordinates
  \[
  dX_t^i = - \frac{1}{2} \sum_{jk} \Gamma_{jk}^i(t, X_t) d \langle
  X^j, X^k \rangle_t
  \]
  up to the differential of a local martingale.
\end{proposition}

\begin{proof}
  This can be proved in the same way as in the case of a fixed
  connection (see e.g.\ \cite[Proposition~3.7]{emery}).
\end{proof}

\begin{example}\label{ex}
  Let $M = \R$ equipped with the standard metric $g_0$, and let $u$ be
  a strictly positive smooth function on $\R_+ \times \R$. Define the
  metric $g(t, \cdot)$ by $g(t,x) = u(t,x) g_0(x)$, and let
  $\nabla(t)$ be its Levi-Civita connection. Let $b$ and $\sigma$ be
  smooth functions on $\R_+ \times \R$, and $X$ the solution of the
  SDE
  \[
  dX_t = b(t,X_t)dt + \sigma(t,X_t) dW_t,
  \]
  where $W$ is a standard one-dimensional Brownian motion. Then $X$ is
  a $(\nabla(t))_{t \geq
    0}$-martingale 
  if and only if
  \[
  b = - \frac{u' \sigma^2}{4u}
  \]
  on $\{(t, X_t) \, | \, t \geq 0\}$ (the prime denotes
  differentiation with respect to~$x$).
\end{example}

\begin{proof}
  Taking into account that the unique Christoffel symbol of
  $\nabla(t)$ equals $u'/(2u)$, the claim follows immediately from
  Proposition~\ref{local}.
\end{proof}

\section{Convergence of martingales}
\subsection{Local convergence}
\begin{proposition}[cf.\ {\cite[Lemma~7.187]{hackenbroch} or
    \cite[Theorem~2.5.6]{hsu}} for the case of a fixed
  connection]\label{locconv}
  Let $U \subseteq M$ be an open subset with the following property:
  There exists a smooth function $\varphi = (\varphi^1, \ldots,
  \varphi^d): M \to \R^d$ such that
  \begin{itemize}
  \item $\varphi|_U$ is bounded,

  \item $\varphi|_U$ is a diffeomorphism onto its image, and

  \item $\Hess^{\nabla(t)} \varphi^i(x) \geq 0$ for all $i \in \{1,
    \ldots, d\}$, all $x \in U$ and all $t \geq 0$.
  \end{itemize}
  Then each $(\nabla(t))_{t \geq 0}$-martingale $X$ converges almost
  surely on the set $\Omega_0 := \{X \mbox{ lies eventually }$
  $\mbox{in } U\}$.
\end{proposition}

\begin{remark}
  In the case of a fixed connection each point $x \in M$ has a
  neighbourhood $U$ with that property (see e.g.\
  \cite[Lemma~7.187]{hackenbroch} or \cite[Theorem~2.5.6]{hsu}).
\end{remark}

\begin{proof}[Proof of Proposition~\rm\ref{locconv}]
  By Definition~\ref{martingale} for each $i \in \{1, \ldots, d\}$
  there exists a real-valued local martingale $M^i$ such that
  \[
  \varphi^i(X_t) = \varphi^i(X_0) + M_t^i + A_t^i,
  \]
  where $A_t^i := \int_0^t \Hess^{\nabla(s)} \!\! \varphi^i(dX_s,
  dX_s)$.

  Since $\Hess^{\nabla(s)} \varphi^i \geq 0$ on $U$, the process $A$
  is eventually non-decreasing and in particular bounded from below on
  $\Omega_0$. Since $\varphi^i|_U$ is bounded, it follows that the
  local martingale $M^i$ is bounded from above and hence convergent on
  $\Omega_0$ (because it is a time-changed Brownian motion). This
  implies that the process $A^i$ is bounded and hence convergent on
  $\Omega_0$ (since it is eventually non-decreasing). Consequently,
  the process $\varphi^i(X)$ converges on $\Omega_0$, and, since
  $\varphi|_U$ is a diffeomorphism onto its image, so does the process
  $X$.
\end{proof}

\subsection{Darling-Zheng}
An important result of martingale theory in the case of a fixed
connection is the convergence theorem of Darling and Zheng (see e.g.\
\cite[Satz~7.190]{hackenbroch}): Let $X$ be an $M$-valued martingale
with respect to a fixed connection $\nabla$, and $g_0$ an arbitrary
Riemannian metric on $M$. Then
\[
\{ X_\infty \mbox{ exists in } M \} \subset \{ \mbox{$\int_0^\infty$}
g_0(dX_s, dX_s) < \infty \} \subset \{ X_\infty \mbox{ exists in }
\hat M \},
\]
where $\hat M$ is the Alexandrov compactification of $M$. In the case
of a time-dependent connection at least the second inclusion does not
hold. To see this consider the following example:

\begin{example}
  In the situation of Example~\ref{ex} take $u(t,x) = \exp(a(t)x)$,
  $\sigma(t,x) = \sigma(t)$ and $b(t,x) = -\frac{1}{4} a(t)
  \sigma^2(t)$ with smooth functions $a, \sigma: \R_+ \to \R$. Then
  $X$ is a $(\nabla(t))_{t \geq 0}$-martingale, and
  \[
  X_t = X_0 - \frac{1}{4} \int_0^t a(s) \sigma(s)^2 ds + \int_0^t
  \sigma(s) dW_s,
  \]
  so that $\int_0^t g_0(dX_s, dX_s) = \int_0^t \sigma(s)^2 ds$.  If
  $\sigma$ is chosen in such a way that $\int_0^\infty \sigma(s)^2 ds
  < \infty$, then $\int_0^\infty g_0(dX_s, dX_s) < \infty$, but the
  function~$a$ (being arbitrary) can be chosen in such a way that $X$
  does not converge in $\hat \R$.
\end{example}

In the Riemannian case one might hope that $\{ \mbox{$\int_0^\infty$}
g(s)(dX_s, dX_s) < \infty \} \subset \{ X_\infty \mbox{ exists in }
\hat M \}$.  This however turns out to be wrong as well:
\begin{example}
  In the situation of Example~\ref{ex} take $u(t,x) = u(t)$,
  $\sigma(t,x) \equiv 1$ and $b(t,x) \equiv 0$. Then $X$ is a
  $(\nabla(t))_{t \geq 0}$-martingale, and
  \[
  X_t = X_0 + W_t,
  \]
  so that $\int_0^t g(s)(dX_s, dX_s) = \int_0^t u(s) ds$.  If $u$ is
  chosen in such a way that $\int_0^\infty u(s) ds < \infty$, then
  $\int_0^\infty g(s)(dX_s, dX_s) < \infty$, but obviously $X$ does
  not converge in $\hat \R$.
\end{example}

\section{Uniqueness of martingales with given terminal value}

\begin{proposition}[cf.~{\cite[Lemma~7.204]{hackenbroch}} for the case of a fixed connection]\label{convex}
  Let $M_0$ be a submanifold of $M$ which is totally geodesic in $M$
  with respect to $\nabla(t)$ for all $t$. Then for each $x_0 \in M$
  and each $T \geq 0$ there exist an open neighbourhood $V$ of $x_0$
  in $M$ and a non-negative function $f \in C^\infty(V)$ satisfying
  \[
  f(x) = 0 \Longleftrightarrow x \in M_0
  \]
  and
  \begin{equation}\label{conv}
    \Hess^{\nabla(s)} f(x) \geq 0 
  \end{equation}
  for all $s \in [0,T]$ and all $x \in V$.
\end{proposition}

\begin{proof}
  Let $d_0 := \dim M_0$.  Choose coordinates $x_1, \ldots, x_d$ for
  $M$ on a neighbourhood $O$ of $x_0$ in such a way that
  \begin{equation}\label{subm}
    O \cap M_0 = O \cap \{x_{d_0+1} = \ldots = x_d = 0\}.
  \end{equation}
  We will show that for sufficiently small $c > 0$ the function
  \[
  f(x) := \frac{1}{2} (c+ |\tilde x|^2) |\hat x|^2,
  \]
  where $\tilde x := (x_1, \ldots, x_{d_0})$ and $\hat x :=
  (x_{d_0+1}, \ldots, x_d)$, does the job on a possibly smaller
  neighbourhood $V$ of $x_0$. All we have to show is that \eqref{conv}
  holds provided one chooses $c$ and $V$ small enough.

  Let $\Gamma_{ij}^k(s,x)$ be the Christoffel symbols with respect to
  $\nabla(s)$. Since $M_0$ is totally goedesic and because of
  \eqref{subm} one has
  \[
  \Gamma_{ij}^k(s,x) = 0, \quad i,j \leq d_0, \quad k \geq d_0+1
  \]
  for all $s \geq 0$ and all $x \in O \cap M$.  By the compactness of
  $[0,T]$ this implies the existence of a constant $C < \infty$ such
  that
  \begin{equation}\label{christoffel}
    |\Gamma_{ij}^k(s,x)| \leq C |\hat x|, \quad i,j \leq d_0, \quad k \geq d_0+1.
  \end{equation}
  Since
  \[
  \Hess^{\nabla(s)} f(x) = \sum_{i,j=1}^d H_{ij}(s,x)\, dx_i \otimes
  dx_j,
  \]
  where
  \[
  H_{ij}(s,x) := \frac{\partial^2 f}{\partial x_i \partial x_j}(x) -
  \sum_{k=1}^d \Gamma_{ij}^k(s,x) \frac{\partial f}{\partial x_k}(x) ,
  \]
  it suffices to show that the matrix $H(s,x)$ is positive definite
  for all $s \in [0,T]$ and all $x \in V \setminus M_0$, provided that
  $c$ and $V$ are chosen small enough.  Using the decomposition of
  $\{1, \ldots, d\}$ into $I = \{1, \ldots, d_0\}$ and $J = \{d_0+1,
  \ldots, d\}$, this is true if and only if the same statement holds
  for the block matrix $H^*(s,x)$ defined by
  \begin{align*}
    H^*(s,x) :=
    \begin{pmatrix}
      \displaystyle\frac{1}{|\hat x|^2}(H_{ij}(s,x))_{(i,j) \in I \times I} & \displaystyle\frac{1}{|\hat x|}(H_{ij}(s,x))_{(i,j) \in I \times J}\\
      \displaystyle\frac{1}{|\hat x|}(H_{ij}(s,x))_{(i,j) \in J \times
        I} & (H_{ij}(s,x))_{(i,j) \in J \times J}
    \end{pmatrix}.
  \end{align*}
  Since
  \[
  \frac{\partial f}{\partial x_k}(x) = \left\{
    \begin{array}{ll}
      x_k |\hat x|^2 & \qquad 1 \leq k \leq d_0,\\
      x_k (c + |\tilde x|^2) & \qquad d_0+1 \leq k \leq d,
    \end{array}
  \right.
  \]
  and
  \[
  \frac{\partial^2 f}{\partial x_i \partial x_j}(x) = \left\{
    \begin{array}{ll}
      \delta_{ij} |\hat x|^2 & \qquad 1 \leq i,j \leq d_0,\\
      2 x_i x_j & \qquad  1 \leq i \leq m \mbox{ and } d_0+1 \leq j \leq n,\\
      \delta_{ij} (c + |\tilde x|^2) & \qquad m+1 \leq i,j \leq n,
    \end{array}
  \right.
  \]
  it is easy to see (using \eqref{christoffel}) that on $[0,T] \times
  (V \setminus M)$ the matrix $H^*(s,x)$ is arbitrarily close to the
  (obviously positive definite) block matrix
  \begin{align*}
    \begin{pmatrix}
      I & 0\\
      0 & cI
    \end{pmatrix}
  \end{align*}
  provided that $c$ and $V$ are chosen small enough.
\end{proof}

\begin{corollary}\label{martsub}
  Let $M_0$ be submanifold of $M$ which is totally geodesic in $M$
  with respect to $\nabla(t)$ for all $t$.  Then given $T > 0$ each
  point $x_0 \in M_0$ has an open neighbourhood $V$ in $M$ with the
  following property: If $X$ is a $V$-valued $(\nabla(t))_{t \geq
    0}$-martingale such that a.s.\ $X_T \in M_0$, then a.s.\ $X_t \in
  M_0$ for all $t \in [0,T]$.
\end{corollary}

\begin{proof}
  Choose $V$ and $f$ as in Proposition~\ref{convex}. Then $f(X)$ is a
  non-negative submartingale with $f(X_T) = 0$~a.s., hence $f(X)
  \equiv 0$~a.s.\ on $[0,T]$.
\end{proof}

\begin{corollary}[Uniqueness of $(\nabla(t))_{t \geq 0}$-martingales with given terminal value] 
  Given $T > 0$ each point $x \in M$ has an open neighbourhood $V$
  with the following property: If $X$ and $Y$ are two $V$-valued
  $(\nabla(t))_{t \geq 0}$-martingales such that a.s.\ $X_T = Y_T$,
  then a.s.\ $X_t = Y_t$ for all $t \in [0,T]$.
\end{corollary}

\begin{proof}
  Apply Corollary~\ref{martsub} to the diagonal embedding of $M$ into
  $M \times M$ equipped with the product connections $\nabla(t)
  \otimes \nabla(t)$.
\end{proof}

\section{Behaviour of semimartingales under
  maps} 
\begin{proposition}[cf.~{\cite[Satz~7.156]{hackenbroch}} for the case
  of fixed connections]\label{transf}
  Let $N$ be another manifold, also equipped with a smooth family of
  connections $(\nablaa(t))_{t \geq 0}$, and let $f \in
  C^\infty({\R_+\times M}, N)$. Let $X$ be a semimartingale on $M$,
  $U$ a $(\nabla(t))_{t \geq 0}$-horizontal or $(g(t))_{t \geq
    0}$-Riemann-horizontal lift of $X$, and $Z$ the corresponding
  $(\nabla(t))_{t \geq 0}$-antidevelopment or $(g(t))_{t \geq
    0}$-Riemann-anti\-deve\-lop\-ment.  Moreover, let $\U$ be a
  $(\nablaa(t))_{t \geq 0}$-horizontal or $(\tilde g(t))_{t \geq
    0}$-Riemann-horizontal lift of the image process $\X_t := f(t,
  X_t)$, and $\Z$ the corresponding $(\nablaa(t))_{t \geq
    0}$-antidevelopment or $(\tilde g(t))_{t \geq
    0}$-Riemann-anti\-de\-ve\-lop\-ment.  Then the following formula
  holds:
  \begin{equation}\label{FormulaZ}
    d\Z_t = \U_t^{-1} \frac{\partial f}{\partial t}(t, X_t) dt + \U_t^{-1} df U_t d Z_t + \frac{1}{2} \U_t^{-1} \Hess^{\nabla(t), \nablaa(t)} \!\! f (t, X_t)(dX_t, dX_t).
  \end{equation}
\end{proposition}

\begin{proof}
  Let $n := \dim N$ and $\varphi \in C^\infty(N)$. Then by
  Corollary~\ref{ito3} and the pullback formula for the quadratic
  variation (see e.g.\ \cite[Satz~7.61]{hackenbroch}),
  \begin{align}
    d \varphi(\X_t)
    & =  \sum_{k = 1}^n (\U_t e_k) \varphi(\X_t) d \Z_t^k {} + \frac{1}{2} \Hess^{\nablaa(t)} \! \varphi(d \X_t, d \X_t) \nonumber \\
    & = \sum_{k = 1}^n (\U_t e_k) \varphi(\X_t) d \Z_t^k {} +
    \frac{1}{2} (f^* \Hess^{\nablaa(t)} \! \varphi) (d X_t, d
    X_t). \label{one}
  \end{align}
  On the other hand, using the Hessian composition formula
  \[
  \Hess^{\nabla(t)}(\varphi \circ f) = d\varphi \circ
  \Hess^{\nabla(t), \nablaa(t)} \! f + f^* \Hess^{\nablaa(t)} \varphi
  \]
  (see e.g.\ \cite[Satz~7.155]{hackenbroch}), we obtain
  \begin{align}
    d \varphi(\X_t)
    & =  d (\varphi \circ f) (t, X_t) \nonumber \\
    & =   \frac{\partial (\varphi \circ f)}{\partial t}(t, X_t) dt + \sum_{i = 1}^d (U_t e_i) (\varphi \circ f)(t, X_t) d Z_t^i + \frac{1}{2} \sum_{i, j=1}^d \Hess^{\nabla(t)} (\varphi \circ f)(dX_t, dX_t) \nonumber \\
    & =   \frac{\partial (\varphi \circ f)}{\partial t}(t, X_t) dt + \sum_{i = 1}^d (U_t e_i) (\varphi \circ f)(t, X_t) d Z_t^i \nonumber \\
    &\quad + \frac{1}{2} (d \varphi \circ \Hess^{\nabla(t),
      \nablaa(t)} \!\! f) (dX_t, dX_t) + \frac{1}{2} (f^*
    \Hess^{\nablaa(t)} \! \varphi) (dX_t, dX_t). \label{two}
  \end{align}
  Combining \eqref{one} and \eqref{two} we obtain
  \begin{align*}
    \sum_{k = 1}^n (\U_t e_k) \varphi(\X_t) d \Z_t^k
    & =  \frac{\partial (\varphi \circ f)}{\partial t}(t, X_t) dt + \sum_{i = 1}^d (U_t e_i) (\varphi \circ f)(t, X_t) d Z_t^i \\
    &\quad + \frac{1}{2} (d \varphi \circ \Hess^{\nabla(t),
      \nablaa(t)} \!\! f) (dX_t, dX_t).
  \end{align*}
  Since this holds for all $\varphi \in C^\infty(N)$, it follows that
  \[
  \sum_{k = 1}^n (\U_t e_k) d \Z_t^k = \frac{\partial f}{\partial
    t}(t, X_t) dt + \sum_{i = 1}^d (U_t e_i) f (t, X_t) d Z_t^i +
  \frac{1}{2} \Hess^{\nabla(t), \nablaa(t)} \!\! f (dX_t, dX_t)
  \]
  and hence
  \[
  d \Z_t = \U_t^{-1} \frac{\partial f}{\partial t}(t, X_t) dt +
  \U_t^{-1} df U_t d Z_t + \frac{1}{2} \U_t^{-1} \Hess^{\nabla(t),
    \nablaa(t)} \!\! f (t, X_t)(dX_t, dX_t). \qedhere
  \]
\end{proof}

\begin{corollary}
  If the connections $\nabla(t)$ are the Levi-Civita connections of
  Riemannian metrics $g(t)$ and if $X$ is a $(g(t))_{t \geq
    0}$-Brownian motion (whose $(g(t))_{t \geq
    0}$-Riemann-antidevelopment $W$ is a Euclidean Brownian motion),
  then
  \begin{align}\label{FormulaZRiem}
    d\Z_t = \U_t^{-1} \left( \frac{\partial f}{\partial t} +
      \frac{1}{2} \Delta^{g(t), \nablaa(t)} f \right) (t, X_t) dt +
    \U_t^{-1} df U_t\, d W_t,
  \end{align}
  where $\Delta^{g(t), \nablaa(t)} u$ is the tension field of $u$ with
  respect to $g(t)$ and $\nablaa(t)$.
\end{corollary}

\begin{corollary}\label{btm}
  The function $f$ maps $(g(t))_{t \geq 0}$-Brownian motions to
  $(\nabla(t))_{t \geq 0}$-martingales if and only if
  \[
  \frac{\partial f}{\partial t} + \frac{1}{2} \Delta^{g(t),
    \nablaa(t)} f = 0
  \]
  for all $t \geq 0$.
\end{corollary}

\begin{remark}
  In the situation of Proposition \ref{transf} one may consider the
  ``intrinsic'' antidevelopments of $X$, respectively $\X$, defined by
$$\A_t:=U_0 Z_t,\quad\text{respectively}\quad\tilde\A_t:=\U_0 \Z_t,$$
which take values in $T_{X_0}^{\mathstrut}M$, respectively~$T_{\tilde
  X_0}N$. Note that
$$d\A_t=\parr_{0,t}^{-1}\circ dX_t,\quad\text{respectively}\quad d\tilde\A_t=\tildeparrinv_{0,t}\circ d\X_t,$$
where $\parr_{0,t}\equiv U_t^{\mathstrut}\,U_0^{-1}$ and
$\tildeparr_{0,t}\equiv \U_t^{\mathstrut}\,\U_0^{-1}$ denote the
parallel transports along $X$, respectively~$\X$.  Then formula
\eqref{FormulaZ} reads more intrinsically as
\begin{equation*}
  d\tilde\A_t = \tildeparrinv_{0,t}\,\frac{\partial f}{\partial t}(t, X_t)\, dt + \tildeparrinv_{0,t}\, df\, \parr_{0,t}^{\mathstrut}\, d\A_t + \frac{1}{2} \tildeparrinv_{0,t}\, \Hess^{\nabla(t), \nablaa(t)} \! f (t, X_t)(dX_t, dX_t). 
\end{equation*}
The same remark applies to formula \eqref{FormulaZRiem} which then
reads as
\begin{align*}
  d\tilde\A_t = \tildeparrinv_{0,t}\,\left( \frac{\partial f}{\partial
      t} + \frac{1}{2} \Delta^{g(t), \nablaa(t)} f \right) (t, X_t) dt
  + \tildeparrinv_{0,t}\, df \, \parr_{0,t}^{\mathstrut}\, d\A_t\,.
\end{align*}
Recall that in this formula $\A_t=U_0W_t$ is a Euclidean Brownian
motion in $T_{X_0}^{\mathstrut}M$.
\end{remark}

\section{Derivative processes, martingales on the tangent bundle and applications to the nonlinear heat equation}
In this section we assume for simplicity that the connections
$\nabla(t)$ are torsion-free.  Let $\nabla'(t)$ the complete and
$\nabla^h(t)$ the horizontal lift of $\nabla(t)$ to the tangent bundle
$TM$. In the same way as in \cite{arnaudonthalmaier} one can obtain
the following results.

\begin{theorem}[cf.\ {\cite[Theorem~3.1]{arnaudonthalmaier}} for the case of a fixed connection]\label{deriv}
  Let $I$ be an open interval containing 0 and $(X_t(s))_{t \geq 0, s
    \in I}$ a $C^1$-family of continuous $M$-valued $(\nabla(t))_{t
    \geq 0}$-martingales. Then the $TM$-valued derivative process
  $(X_t')_{t \geq 0}$ defined by
$$X_t' := \left.\frac{\partial}{\partial s}\right|_{s=0} X_t(s)$$ is a $(\nabla'(t))_{t \geq 0}$-martingale.
\end{theorem}

\begin{theorem}[cf.\ {\cite[Corollary~4.4]{arnaudonthalmaier}} for the case of a fixed connection]\label{horliftmart}
  A $TM$-valued semimartingale $J$ is a $(\nabla^h(t))_{t \geq
    0}$-martingale if and only if
  \begin{enumerate}
  \item its projection $X$ to $M$ is a $(\nabla(t))_{t \geq
      0}$-martingale, and

  \item $d(\parr_{0,t}^{-1} J_t) \stackrel{\m}{=}
    0$. 
  \end{enumerate}
\end{theorem}

\begin{theorem}[cf.\ {\cite[Theorem~4.12]{arnaudonthalmaier}} for the case of a fixed connection]\label{compmart}
  A $TM$-valued semimartingale $J$ is a $(\nabla'(t))_{t \geq
    0}$-martingale if and only if
  \begin{enumerate}
  \item its projection $X$ to $M$ is a $(\nabla(t))_{t \geq
      0}$-martingale, and

  \item $d(\Theta_{0,t}^{-1} J_t) \stackrel{\m}{=} 0$, where
    $\Theta_{0,t}: T_{X_0}M \to T_{X_t}M$ denotes the damped parallel
    transport along~$X$, defined by the covariant equation
    \[
    d \left( \parr_{0,t}^{-1} \Theta_{0,t} \right) =
    -\frac{1}{2} \parr_{0,t}^{-1} R^{\nabla(t)}(\Theta_{0,t},
    dX_t)dX_t, \qquad \Theta_{0,0} = \Id_{T_{X_0}M}.
    \]
  \end{enumerate}
\end{theorem}

\begin{remark}
  In the Riemannian case the condition $d(\parr_{0,t}^{-1} J_t)
  \stackrel{\m}{=} 0$ in Theorem~\ref{horliftmart} can also be
  expressed using the Riemann-parallel transport
  $\parr_{0,t}^{\Riem}$; using Remark~\ref{relpar} one obtains that it
  is equivalent to
$$d \left( (\parr_{0,t}^{\Riem})^{-1} J_t \right) \stackrel{\m}{=} \frac{1}{2} (\parr_{0,t}^{\Riem})^{-1} \left( \frac{\partial g}{\partial t} \right)^\# \! J_t \, dt.$$ Similarly, the equation defining the damped parallel transport is equivalent to
\begin{equation}\label{damped2}
  d \left( (\parr_{0,t}^{\Riem})^{-1} \Theta_{0,t} \right) = \frac{1}{2} (\parr_{0,t}^{\Riem})^{-1} \left( \frac{\partial g}{\partial t} \right)^\# \! J_t \, dt - \frac{1}{2} (\parr_{0,t}^{\Riem})^{-1} R^{g(t)}(\Theta_{0,t}, dX_t)dX_t.
\end{equation}
If $X$ is a $(g(t))_{t \geq 0}$-Brownian motion, \eqref{damped2}
simplifies to
\begin{equation}\label{dampedbrownian}
  d \left( \parr_{0,t}^{-1} \Theta_{0,t} \right) = \frac{1}{2} \left(\parr_{0,t}^{\Riem}\right)^{-1} \left( \left( \frac{\partial g}{\partial t} \right)^\# - (\Ric^{g(t)})^\# \right) \Theta_{0,t} dt,
\end{equation}
which coincides with the expression given in
\cite[Definition~2.1]{ACT} and \cite[Definition~3.1]{Coulibaly}.
\end{remark}

Combining Theorems~\ref{deriv} and \ref{compmart} one obtains

\begin{corollary}\label{dermart}
  Let $I$ be an open interval containing $0$, $(X_t(s))_{s \in I}$ a
  $C^1$-family of continuous $M$-valued martingales, $X_t := X_t(0)$,
  and $(X_t')_{t \geq 0}$ the $TM$-valued derivative process defined
  by $X_t' := \frac{\partial}{\partial s}|_{s=0} X_t(s)$.
  Then the process $(\Theta_{0,t}^{-1} X_t')_{t \geq 0}$ is a
  $T_{X_0}M$-valued local martingale.
\end{corollary}

Let now $N$ another differentiable manifold and $T_1 < T_2$. Let
$(g(t))_{T_1 \leq t \leq T_2}$ be a smooth family of Riemannian
metrics on $M$, $\nabla(t)$ the Levi-Civita connection of $g(t)$ and
$(\tilde \nabla(t))_{T_1 \leq t \leq T_2}$ a smooth family of
connections on $N$. Let $u: [T_1,T_2] \times M \to N$ be a solution of
the non-linear heat equation
\begin{equation}\label{NonlinHeatEqu}
  \frac{\partial u}{\partial t}  = \frac{1}{2} \Delta^{g(t), \tilde \nabla(t)} u.
\end{equation}
We fix $x \in M$, let $(X_t)_{0 \leq t \leq T_2 -T_1}$ be an
$M$-valued $(g(T_2-t))_{0 \leq t \leq T_2-T_1}$-valued Brownian motion
starting at $x$, and define
$$\X_t := u(T_2-t, X_t),\quad 0\leq t \leq T_2 -T_1.$$ 

\begin{proposition}[cf.\ {\cite[(5.22)]{arnaudonthalmaier}} for the case of fixed Riemannian metrics]\label{martheat}
  Let $$u: [T_1,T_2] \times M \to N$$ be a solution of
  Eq.~\eqref{NonlinHeatEqu}.  Let $\Theta_{0,t}: T_xM \to T_{X_t}M$ be
  the damped parallel transport along~$X$, and $\tilde \Theta_{0,t}:
  T_{\X_0}N \to T_{\X_t}N$ the damped parallel transport along $\X$,
  where $X$ and $\X$ are defined as above.  Then for each $v \in T_xM$
  the $T_{u(T,x)}N$-valued process
  \[
  \tilde \Theta_{0,t}^{-1}\, du(T_2-t, X_t)\, \Theta_{0,t}\, v,\quad
  0\leq t \leq T_2 -T_1,
  \]
  is a local martingale.
\end{proposition}

\begin{proof}
  Let $\gamma: \R \to M$ be a smooth curve with $\gamma(0) = x$ and
  $\dot \gamma(0) = v$. 
  By \cite[Theorem~3.1]{act2} there exists a smooth family
  $(X_t(s))_{0 \leq t \leq T_2-T_1,\, s \in \R}$ of $M$-valued
  $(g(T_2-t))_{0 \leq t \leq T_2-T_1}$-Brownian motions satisfying
  $X_0(s) = \gamma(s)$ for all $s \in \R$, $X_t(0) = X_t$ for all $t
  \in [0,T_2-T_1]$ and
  \[
  \left. \frac{\partial}{\partial s} \right|_{s=0} X_t(s) =
  \Theta_{0,t} v.
  \]
  Let
  \[
  \X_t(s) := u(T_2-t, X_t(s)).
  \]
  By Corollary~\ref{btm} the process $(\X_t(s))_{0 \leq t \leq
    T_2-T_1}$ is an $N$-valued $(\tilde \nabla(t))_{0 \leq t \leq
    T_2-T_1}$-martingale for each $s \in \R$. Moreover,
  \[
  \left. \frac{\partial}{\partial s} \right|_{s=0} \X_t(s) = du(T_2-t,
  X_t) \left. \frac{\partial}{\partial s} \right|_{s=0} X_t(s) =
  du(T_2-t, X_t) \Theta_{0,t} v.
  \]
  Therefore the result follows immediately from
  Corollary~\ref{dermart}.
\end{proof}

\begin{remark}\label{repform}
  If the local martingale in Proposition~\ref{martheat} is a true
  martingale, we obtain the stochastic representation formula
  \begin{equation}\label{repr}
    du(T_2, x) = E \left[ \tilde \Theta_{0,T_2-T_1}^{-1} du(T_1, X_{T_2-T_1}) \Theta_{0,T_2-T_1}^{\mathstrut} \right].
  \end{equation}
\end{remark}

\begin{theorem}\label{dampedest}
  Let $M$ be a connected differentiable manifold equipped with a
  smooth family $(g(t))_{-\infty < t \leq T}$ of Riemannian metrics
  satisfying
  \begin{equation}\label{superricci}
    \frac{\partial g}{\partial t} + \Ric_{g(t)} \geq K > 0
  \end{equation}
  (uniformly strict super Ricci flow), and let $(N, \tilde g)$ be a
  Riemannian manifold of nonpositive sectional curvature.  Then every
  ancient solution $u: [-\infty, T] \times M \to N$ of the non-linear
  heat equation
  \[
  \frac{\partial u}{\partial t} = \frac{1}{2} \Delta^{g(t), \tilde g}
  u
  \]
  whose differential is bounded is constant.
\end{theorem}

\begin{proof}[Proof of Theorem~\rm\ref{dampedest}]
  The curvature conditions imply that $\|\Theta_{0,s}\| \leq
  e^{-K_1s/2}$ and $\|\tilde \Theta_{0,s}^{-1}\| \leq 1$, so that the
  local martingale in Proposition~\ref{martheat} is bounded and hence
  a true martingale. The representation formula \eqref{repr} then
  implies that
  \[
  \|du(t, x)\|_{\tilde g} \leq e^{-Ks/2} \sup_{y \in M} \|du(t-s,
  y)\|.
  \]
  The claim now follows from letting $s \to \infty$.
\end{proof}

\begin{remark}
  More refined representation formulas and Liouville theorems for the
  nonlinear heat equation in the spirit of \cite{thalmaierwang} will
  be derived in a subsequent paper.
\end{remark}

\end{document}